\documentclass[oneside,reqno]{amsart}
\usepackage{amsmath,amssymb,amsfonts,fancyhdr, verbatim}
\usepackage{comment}
\usepackage{amsrefs}
\usepackage{comment}
\usepackage[toc,page]{appendix}
\usepackage{verbatim}
\immediate\write18{texcount -tex -sum  \jobname.tex > \jobname.wordcount.tex}

\def\imod#1{\allowbreak\mkern5mu({\operator@font mod}\,\,#1)}
\makeatother

\def\F{{\mathfrak{L}}}

\def\w{{\omega}}
\def\M{{\mathcal M}}

\newcommand{\h}[1]{\left( #1\right)^{\wedge}}
\newcommand{\hb}[1]{\overline{#1}}

\newtheorem{lem}{Lemma}

\newtheorem{cor}{Corollary}

\newtheorem{thm}{Theorem}
\numberwithin{equation}{section}

\begin{document}
\title[Bounds on Autocorrelation Coefficients of Rudin-Shapiro Polynomials]{Bounds on Autocorrelation Coefficients of Rudin-Shapiro Polynomials}

\author{J.-P. Allouche }
\address{Jean-Paul Allouche \\
CNRS, IMJ-PRG, Sorbonne Universit\'e  \\
4 Place Jussieu, Case 247  \\
F-75252 Paris Cedex 05 \\ France}
\email{jean-paul.allouche@imj-prg.fr}

\author{S. Choi}
\address{Stephen Choi \\
Department of Mathematics \\
Simon Fraser University \\
8888 University Drive \\
Burnaby, British Columbia V5A  1S6 \\
Canada}
\email{schoia@sfu.ca}

\author{A. Denise}
\address{Alain Denise \\
LRI et I2BC, Universit\'e Paris-Sud, CNRS \\
Universit\'e Paris-Saclay \\
F-91405 Orsay Cedex}
\email{alain.denise@lri.fr}

\author{T. Erd\'{e}lyi}
\address{Tam\'{a}s Erd\'{e}lyi \\
Department of Mathematics \\  Texas A\&M University, \\
College Station, TX 77843-3368, \\ USA}
\email{terdelyi@math.tamu.edu}

\author{B. Saffari}
\address{Bahman Saffari \\
D\'epartment de Math\'ematiques  \\
Facult\'e des Sciences d'Orsay, Universit\'e Paris-Sud  \\
F-91405 Orsay Cedex \\ France }
\email{Bahman.Saffari@math.u-psud.fr.}

\date{\today}

\maketitle

\begin{flushright}
    \thispagestyle{empty}
    \vspace*{2mm}
    Dedicated to the memory of Professor Jean-Pierre Kahane.
    \vspace*{7mm}
\end{flushright}

\begin{abstract}
We study the autocorrelation coefficients of the Rudin-Shapiro polynomials,
proving in particular that their maximum on the interval $[1, 2^n)$ is bounded from
below by $C_1 2^{\alpha n}$ and is bounded from above by $C_2 2^{\alpha' n}$ where
$\alpha = 0.7302852\cdots$ and $\alpha' = 0.7302867\cdots$.
\end{abstract}

\vspace{3mm}

 \small	
  \textbf{\textit{Keywords --- }} {Golay-Rudin-Shapiro polynomials, autocorrelation coefficients, trigonometric  polynomials}

\section{Introduction}

The Rudin-Shapiro polynomials were discovered by H. S. Shapiro in 1951 \cite{Sha} (also see
the paper of M. J. E. Golay the same year, and the footnote on the first page of \cite{BM}). They were 
studied later by W. Rudin in a 1959 paper \cite{Ru} as recalled, e.g., in \cite{BC}. These 
polynomials, also called the Shapiro-Rudin polynomials, are constructed as follows.

Let $P_0(x)=1$ and $Q_0=1$. For any integer $n \ge 1$, define
\begin{equation}
\label{1}
\begin{cases}
P_n(z):=P_{n-1}(z)+z^{2^{n-1}}Q_{n-1}(z), \\ Q_n(z):=P_{n-1}(z)-z^{2^{n-1}}Q_{n-1}(z).
\end{cases}
\end{equation}
Polynomials $P_n(z)$ and $Q_n(z)$ are called Rudin-Shapiro polynomials. Note that $P_n(z)$ and $Q_n(z)$ are polynomials with $\pm 1$ coefficients of degree $L_n-1$ where $L_n=2^n$.  It is well-known that 
\begin{equation}
\label{2}
\left| P_n (e^{it})\right|^2 + \left| Q_n(e^{it})\right|^2 = L_{n+1}=2^{n+1}, \quad \mbox{for any $t\in [0, 2\pi )$}.
\end{equation}

By the Fourier coefficients of $f$ at $k$, we mean the coefficient for the term $z^k$, or simply
\[
\widehat{f}(k)=\h{f}(k):=\frac{1}{2\pi}\int_{0}^{2\pi} f(e^{it})e^{-ikt}dt.
\]

By the definition we have
\begin{eqnarray}
|P_n|^2(z) & = & z^{2^{n-1}}\hb{P_{n-1}(z)}Q_{n-1}(z)+\hb{z}^{2^{n-1}}P_{n-1}(z)\hb{Q_{n-1}(z)}+2^n, \nonumber \\
\left(\hb{P_n}Q_n\right) (z) &=& 2|P_{n-1}(z)|^2 - z^{2^{n-1}}\hb{P_{n-1}(z)}Q_{n-1}(z) \nonumber \\ 
&  &\qquad \qquad \qquad +\hb{z}^{2^{n-1}}P_{n-1}(z)\hb{Q_{n-1}(z)}-2^n, \label{3} \\
\left(P_n\hb{Q_n}\right) (z) &=& 2|P_{n-1}(z)|^2 + z^{2^{n-1}}\hb{P_{n-1}(z)}Q_{n-1}(z)\nonumber \\
&  & \qquad \qquad \qquad -\hb{z}^{2^{n-1}}P_{n-1}(z)\hb{Q_{n-1}(z)}-2^n. \nonumber 
\end{eqnarray}

We are interested in estimating $\max_{k} |\h{|P_n|^2}(k)|$. If we write 
\[
|P_n(e^{it})|^2 =  \sum_{k=-L_n+1}^{L_n-1}{a_kz^k},
\]
then 
\[
\begin{cases}
 \h{|P_n|^2}(k)=a_k, & \quad \mbox{ when }  - L_n+1 \le k \le L_n-1, \\  
 \h{|P_n|^2}(k)=0, & \quad \mbox{ when }   |k| > L_n.
\end{cases}
\]

We will prove
\begin{thm}
\label{thm 1}
If $P_n$ and $Q_n$ are the $n$-th Rudin-Shapiro polynomials and
$$|P_n(z)|^2 = \sum_{k=-L_n+1}^{L_n-1}{a_kz^k}, \quad |z|=1$$
($a_0=L_n$, $a_k = a_{-k}$, $k \geq 1$), then
$$ C_1\lambda^{n/2}\le \max_{1 \leq k \leq L_n-1}{|a_k|} \leq C_2((1.00000100000025) \lambda) ^{n/2}$$
with absolute constants $C_1,C_2 > 0$ where
\[
\lambda := \frac{(71+6\sqrt{177})^{1/3}+ (71-6\sqrt{177})^{1/3}+5}{3}=2.75217177\cdots
\]
is the real root of $x^3-5x^2+12x-16=0$.  

Also, if we  let
$$\left(\hb{P_n}Q_n\right)(z)= \left({P_n}\hb{Q_n}\right)(1/z) = \sum_{k=-L_n+1}^{L_n-1}{b_kz^k}, \quad |z|=1$$
($b_0=2-L_n$, $b_k = \hb{b_{-k}}$, $k \geq 1$), then
$$ C_3\lambda^{n/2} \le \max_{1\leq k \leq L_n-1}{|b_k|} \leq C_2((1.00000100000025) \lambda) ^{n/2}$$
with an absolute constant $C_3 > 0$. 
\end{thm}

\begin{cor} Under the same assumptions as in  Theorem \ref{thm 1}, we have
$$C_12^{\left(\frac{\log \lambda}{2\log2}\right)n}  < \max_{1 \leq k \leq L_n-1}{|a_k|} < C_22^{(0.7302867\cdots )n}$$
and
$$C_32^{\left(\frac{\log \lambda}{2\log2}\right)n} < \max_{1 \leq k \leq L_n-1}{|b_k|} < C_22^{(0.7302867\cdots )n}.$$
Note that $\frac{\log \lambda}{2\log2}= 0.7302852\cdots$. 
\end{cor}

We remark that the constants $C_1, C_2$ and $C_3$ are absolute and effective and the numerical values of these constants can be obtained by determining the implicit constants in the proofs. 

As an application of Theorem \ref{thm 1} we employ an old result of Littlewood,
see Lemma 5 in Section 4. The following theorem tells much about the
oscillation of the modulus of the Rudin-Shapiro polynomials around
their normalized $L_2$ norm over the unit circle. For further investigations
in this direction see \cite{E1} and \cite{E2}.
\begin{thm}
\label{thm 4}
If $P_n$ and $Q_n$ are the $n$-th Rudin-Shapiro polynomials, 
$$R(t) := |P_n(e^{it})|^2 = \sum_{k=-L_n+1}^{L_n-1}{a_ke^{ikt}}$$
($a_0=L_n$, $a_k = a_{-k}$, $k \geq 1$), and
$$\max_{1 \leq k \leq L_n-1}{|a_k|} \leq C2^{(\lambda_0 - \varepsilon )n}$$
with an absolute constants $C > 0$, $\lambda_0 \geq 1/2$ and $\varepsilon > 0$, then
there are absolute constants $A>0$ and $B>0$ such that the equation 
$R(t) = (1+\eta)2^n$ has at least $A2^{(2 - 2\lambda_0)n}$ distinct solutions in $(-\pi , \pi)$ whenever $\eta \in [-B,B]$ and $n$ is sufficiently large.
\end{thm}

In view of Theorem \ref{thm 1}, we have
\begin{cor}
There is absolute constant $A>0$  such that the equation $R(t) = (1+\eta )2^n$ has at least $A2^{0.5394282n}$ distinct solutions in $(-\pi , \pi)$ whenever $|\eta | \le  2^{-8}$.
\end{cor}

\section{Upper  bound for the autocorrelation coefficients}

M. Taghavi in \cite{Ta1} claimed
$$\max_{1 \leq k \leq L_n-1}{|a_k|} \leq (3.2134)2^{(0.7303)n}\,.$$
However, as Allouche and Saffari observed, in his proof Taghavi used an
incorrect statement saying that the spectral radius of the product of some
matrices is independent of the order of the factors (see the review by the first named author
in Zentralblatt 0921.11042). So what he ended
up with cannot be viewed as a correctly proved result.  We obtain a better upper bound here.

The scheme of the proof of Theorem 1 is as follows. Starting from the recursive relations of Rudin-Shapiro polynomials, we develop recursive relations of the autocorrelation coefficients of $|P_n|^2, \hb{P_n}Q_n$ and $P_n\hb{Q_n}$
 (see $\omega_n$ below) in Lemma 1. We find that the multiplying matrix in each inductive step comes from four $3\times 3$ matrices $A, B, C$ and $D$ (defined below).  However, after applying $n$ inductive steps, the resulting matrices $M$ (a product of $n$ matrices) have $4^n$ different possible matrices which is difficult too handle. By studying the products of $A, B, C$ and $D$ carefully, we show in Lemma 2 that these $4^n$ matrices $M$  are basically generated multiplicatively by two matrices $M_1$ and $B$ only  and $B$ is of order $4$. This nice factorization of $M$ makes us successfully  estimating the spectral norm of $M$ induced by $L_2$ norm in Theorem 3 and hence gives an estimates for the autocorrelation coefficients in Theorem 1.

By induction on \eqref{3}, we have $$\h{|P_n|^2}(2k)=\h{\hb{P_n}Q_n}(2k)=\h{P_n\hb{Q_n}}(2k)=0.$$
For $n \ge 1$, let $S_n$ be the set of all odd integers $k$ with $-L_n < k  < L_n$ and  let
\[
S_n^\tau:=\left\{ k \in S_n : (\tau -3)2^{n-1} < k \le (\tau -2)2^{n-1} \right\}
\]
so that $S_n$ is the disjoint union of $S_n^1, S_n^2, S_n^3$ and $S_n^4$.  (The definition of $S_n^\tau$ is slightly different in \cite{Ta1} in order to make $S_1=\{ -1, +1 \}$ is a union of 
$S_1^1, S_1^2, S_1^3$ and $S_1^4$.)

For $n \ge 2$, let $k_n$ be an odd integer in $S_n$. We define $k_{n-1}$ and $k_n'$ from $k_n$ as follows
\begin{equation}
k_n' := 
\begin{cases} 
k_n+2^n & \mbox{ if $k_n \in S_n^1 \cup S_n^2$,} \\
k_n-2^n & \mbox{ if $k_n \in S_n^3 \cup S_n^4$,} 
\end{cases}
\label{4}
\end{equation}
and
\begin{equation}
k_{n-1} := 
\begin{cases} 
k_n & \mbox{ if $k_n \in S_n^2 \cup S_n^3$,} \\
k_n' & \mbox{ if $k_n \in S_n^1 \cup S_n^4$.} 
\end{cases}
\label{5}
\end{equation}
It is easily to see that if $k_n \in S_n$, then $k_{n-1} \in S_{n-1}$ and $k_n' \in S_n$.

Let
\[
\omega_n(k_n) := \left( \begin{array}{c}\h{|P_n|^2}(k_n) \\ \h{\hb{P_n}Q_n}(k_n') \\ \h{P_n\hb{Q_n}}(k_n') \end{array}\right) .
\]
Lemma 1 below gives a recursive relation of $\omega_n(k_n)$. 
\begin{lem} For $n\ge 2$, we have
\begin{equation}
\label{6}
\w_n(k_n)= \M\w_{n-1}(k_{n-1}),
\end{equation}
where $\M$ is one of the following four $3\times 3$ matrices, whenever $k_n$ is in $S_n^1, S_n^2, S_n^3$ and $S_n^4$ respectively,
\[
A=
\begin{pmatrix}
0 & 0 & 1 \\
2 & -1 & 0 \\
2 & 1 & 0
\end{pmatrix} ,
B=
\begin{pmatrix}
0 & 0 & 1 \\
0 & -1 & 0 \\
0 & 1 & 0
\end{pmatrix} ,
C=
\begin{pmatrix}
0 & 1 & 0 \\
0 & 0 & 2 \\
0 & 0 & -1
\end{pmatrix} ,
D=
\begin{pmatrix}
0 & 1 & 0 \\
2 & 0 & 1 \\
2 & 0 & -1
\end{pmatrix} .
\]
\end{lem}
\begin{proof}
This is Lemma 1 in \cite{Ta1} but there are some typos in the original paper . Let record a correct version here.
First note that 
\begin{eqnarray}
&  & k_n \in S_n^1\cup S_n^3 \Rightarrow k_{n-1} \in S_{n-1}^3\cup S_{n-1}^4, \nonumber \\
&  & k_n \in S_n^2\cup S_n^4 \Rightarrow k_{n-1} \in S_{n-1}^1\cup S_{n-1}^2. \label{7}
\end{eqnarray}
Let $k_n \in S_n^1$. By  \eqref{4} and \eqref{5}, $k_{n-1}=k_n'=k_n+2^n$, so \eqref{7}, together with \eqref{4} and \eqref{5} again, imply that $k_{n-1}'=k_{n-1}-2^{n-1}=k_n+2^{n-1}$.

Now using \eqref{2}, we have
\begin{eqnarray*}
         \h{|P_n|^2}(k_n) 
& = & \h{z^{2^{n-1}}\hb{P_{n-1}}Q_{n-1}}(k_n)+ \h{\hb{z}^{2^{n-1}}P_{n-1}\hb{Q_{n-1}}}(k_n) \\
& = & \h{\hb{P_{n-1}}Q_{n-1}}(k_n-2^{n-1})+ \h{P_{n-1}\hb{Q_{n-1}}}(k_n+2^{n-1})  \\
& = & \h{\hb{P_{n-1}}Q_{n-1}}(k_n-2^{n-1})+ \h{P_{n-1}\hb{Q_{n-1}}}(k_{n-1}')  .
\end{eqnarray*}
As $k_n -2^{n-1} < 1 - 2^{n-1}$, by \eqref{3}, we have $\h{\hb{P_{n-1}}Q_{n-1}}(k_n-2^{n-1})=0$ 
which implies that 
\begin{equation}
\label{8}
 \h{|P_n|^2}(k_n)  = \h{P_{n-1}\hb{Q_{n-1}}}(k_{n-1}')  .
\end{equation}
Next using \eqref{2}, we have
\begin{eqnarray*}
&   &         \h{\hb{P_{n}}Q_{n}}(k_n') \\
& = & 2\h{|P_{n-1}|^2}(k_n') - \h{z^{2^{n-1}}\hb{P_{n-1}}Q_{n-1}}(k_n')   -  \h{\hb{z}^{2^{n-1}}P_{n-1}\hb{Q_{n-1}}}(k_n') \\
& = & 2\h{|P_{n-1}|^2}(k_{n-1}) - \h{\hb{P_{n-1}}Q_{n-1}}(k_n'-2^{n-1})   -  \h{P_{n-1}\hb{Q_{n-1}}}(k_n'+2^{n-1})  .
\end{eqnarray*}
(Note that there are some typos  for the above formulas in \cite{Ta1}.)

Since $k_n'+2^{n-1}>2^{n-1}-1$, by \eqref{3}, $\h{P_{n-1}\hb{Q_{n-1}}}(k_n'+2^{n-1})=0$ and as $k_n'-2^{n-1}=k_n+2^{n-1}=k_{n-1}'$, we have
\begin{equation}
\label{9}
  \h{\hb{P_{n}}Q_{n}}(k_n') = 2\h{|P_{n-1}|^2}(k_{n-1})  -  \h{\hb{P_{n-1}}Q_{n-1}}(k_{n-1}')  .
\end{equation}
A similar argument also gives
\begin{equation}
\label{10}
  \h{P_{n}\hb{Q_{n}}}(k_n') = 2\h{|P_{n-1}|^2}(k_{n-1})  + \h{\hb{P_{n-1}}Q_{n-1}}(k_{n-1}')  .
\end{equation}
Now \eqref{8}, \eqref{9} and \eqref{10} imply that in \eqref{6}, $\M=A$, whenever $k_n \in S_n^1$. Similar calculations yield \eqref{6} for the other three cases, which completes the proof of the lemma.
\end{proof}

Applying Lemma 1 inductively on $n$ , we get
\begin{equation}
\label{11}
\w_n(k_n) = \M_{n-1}\cdots \M_2 \M_1\w_1(k_1)
\end{equation}
with $\w_1(k_1)= \pm \begin{pmatrix} 1 \\ 1 \\ -1 \end{pmatrix}, k_1 = \pm 1, $ and 
\begin{equation}
\label{12}
\M_i \in T_1:=\{ A, B, C, D \} \quad i=1, \ldots ,n-1
\end{equation}
where $T_n$ is the set of all matrices which are a product of $n$ matrices in $\{ A, B, C, D \} $.

We consider any product of two consecutive terms in the matrix product \eqref{11}. Let $j \in \{ 1, \ldots , n-2 \}$, then 
\begin{equation}
\M_j\M_{j+1} \in T_2:=\left\{ AC, AD, B^2, BA, C^2, CD, DA, DB \right\}.
\label{13}
\end{equation}
To see this, suppose that $\M_j=A$, then $k_{n-j}\in S_{n-j}^1$ and so, by \eqref{7}, 
\[
k_{n-(j+1)} \in S_{n-(j+1)}^3 \cup S_{n-(j+1)}^4.
\]
Therefore, $\M_{j+1}\in \{ C, D \}$, that is, $\M_j\M_{j+1}$ is either $AC$ or $AD$. Similar arguments yield \eqref{13} for the cases that $\M_j$ is $B, C$ or $D$. 
In other words, 
\begin{enumerate}
\item[(i)] if $\M_j=A, C$, then $\M_{j+1}=C, D$ so that it produces $AC, AD, CC, CD$,
\item[(ii)] if $\M_j=B, D$, then $\M_{j+1}=A, B$ so that it produces $BA, BB, DA, DB$.
\end{enumerate}
Due to \eqref{7}, $T_2$ does not contain all 16 possible matrices but only 8 matrices. 

We next will show that every element in $T_n$ can be generated by only two matrices $B$ and $M_1$ in a very simple fashion (see Lemma \ref{lem 2}).

We  let
\[
T=
\begin{pmatrix}
1 & 0 & 0 \\
0 & 0 & 1 \\
0 & 1 & 0
\end{pmatrix},  \quad 
I=
\begin{pmatrix}
1 & 0 & 0 \\
0 & 1 & 0 \\
0 & 0 & 1
\end{pmatrix}  .
\]
Then we have
\[
T^2=I, \quad C =TBT \quad \mbox{ and } \quad D=TAT.
\]
Let
\begin{equation}
M_1=AT, \quad M_2=TA, \quad M_3=BT, \quad M_4=TB .
\label{30}
\end{equation}
Then we can write $T_2$ in terms of $M_j$:
\begin{eqnarray*}
T_2 
&=& \left\{ AC, AD, B^2, BA, C^2, CD, DA, DB \right\} \\
&=& \left\{ (AT)(BT), (AT)(AT), (BT)(TB), (BT)(TA), (TB)(BT), \right. \\
& & \qquad \qquad  \left. (TB)(AT), (TA)(TA), (TA)(TB) \right\} \\
&=& \left\{ M_1M_3, M_1M_1, M_3M_4, M_3M_2,M_4M_3, M_4M_1, M_2M_2, M_2M_4 \right\} .
\end{eqnarray*}
In view of (i) and (ii),  we have the following rules when we multiply $M_j$ together:
\begin{enumerate}
\item[(a)] After $M_1$, it is either $M_1$ or $M_3$ only;
\item[(b)] After $M_2$, it is either $M_2$ or $M_4$ only;
\item[(c)] After $M_3$, it is either $M_2$ or $M_4$ only;
\item[(d)] After $M_4$, it is either $M_1$ or $M_3$ only.
\end{enumerate}

\vspace{3mm}

We now assume $n$ is even and consider the possible factorization of $M$ in $T_n$. 

\vspace{3mm}

For any $M \in T_{n}$, then $M$ is a product of $n/2$ matrices, $M_iM_j\in T_2$. Then 
\begin{enumerate}
\item[(1)] If $M$ is starting from $M_1$, then in view of (a)-(d), $M$ is  one of these forms:
\begin{enumerate}
\item[(i)] $M=M_1^{\ell_1}$ for $\ell_1\ge 1$;
\item[(ii)] $M=M_1^{\ell_{1}}M_3$ for $\ell_1 \ge 1$;
\item[(iii)]  $M=M_1^{\ell_1}M_3M_4, \ell_1 \ge 1$;
\item[(iv)] $M=M_1^{\ell_1}M_3M_2^{\ell_{2}}$ for $\ell_1 \ge 1, \ell_2 \ge 1$; 
\item[(v)] $M=M_1^{\ell_{1}}M_3M_2^{\ell_{2}}M_4$ for $\ell_1 \ge 1, \ell_2 \ge 1$;
\item[(vi)] $M=(M_1^{\ell_{1}}M_3M_2^{\ell_{2}}M_4)\;M^*\cdots$ , for $\ell_1 \ge 1, \ell_2 \ge 0$, $M^*=M_1$ or $M_3$. 
\end{enumerate}
\item[(2)] If $M$ is starting from $M_2$, then in view of (a)-(d), $M$  is   one of these forms:
\begin{enumerate}
\item[(i)] $M=M_2^{\ell_2}$, for $\ell_2 \ge 1$;
\item[(ii)]  $M=M_2^{\ell_2}M_4$, for $\ell_2 \ge 1$;
\item[(iii)] $M=(M_2^{\ell_2}M_4)\;M^*\cdots$, for $\ell_2 \ge 1$, $M^*=M_1$ or $M_3$. 
\end{enumerate}
\item[(3)] if $M$ is starting from $M_3$, then in view of (a)-(d), $M$  is   one of these forms:
\begin{enumerate}
\item[(i)] $M=M_3M_2^{\ell_2}$, for $\ell_2 \ge 1$;
\item[(ii)]  $M=M_3M_2^{\ell_2}M_4$, for $\ell_2 \ge 1$;
\item[(iii)] $M=M_3M_4$;
\item[(iv)] $M=(M_3M_2^{\ell_2}M_4)\;M^*\cdots$, for $\ell_2 \ge 0$, $M^*=M_1$ or $M_3$. 
\end{enumerate}
\item[(4)] if $M$ is starting from $M_4$, then in view of (a)-(d),  $M$  is   one of these forms:
\begin{enumerate}
\item[(i)] $M=M_4$;
\item[(ii)] $M=(M_4)\;M^*\cdots$ for $M^*=M_1$ or $M_3$. 
\end{enumerate}
\end{enumerate}

We then rewrite $M$ as a string of block: $M=B_1B_2\cdots B_r$ by the appearance of $M_4$ so that  the first $r-1$ blocks are all ending with $M_4$
\[
B_k=(M_{i_1}M_{i_2}\cdots M_{i_t}M_4)
\]
for $1 \le k \le r-1$ and 
\[
B_r=(M_{i_1}M_{i_2}\cdots M_{i_t})M_4 \quad \mbox{ or } \quad B_r=(M_{i_1}M_{i_2}\cdots M_{i_t})
\]
where $M_{i_j} \in \{ M_1, M_2, M_3\}$.  

\vspace{3mm}

\noindent For the first block $B_1$, $B_1$ must be ending with $M_4$. According to the first matrix of $B_1$, we have the following 4 cases.
\begin{enumerate}
\item[$\bullet$] If $B_1$ is starting from $M_1$, then by (iii) and (v) of (1), $B_1$ is either $M_1^{\ell_1}M_3M_4$ or $M_1^{\ell_{1}}M_3M_2^{\ell_{2}}M_4.$  
\item[$\bullet$] If $B_1$ is starting from $M_2$, then by (ii) of (2), $B_1$ must be $M_2^{\ell_2}M_4.$  
\item[$\bullet$] If $B_1$ is starting from $M_3$, then by (ii)  and (iii) of (3), $B_1$ is either $M_3M_2^{\ell_{2}}M_4$ or $M_3 M_4.$
\item[$\bullet$] If $B_1$ is starting from $M_4$, then by (i) of (4), $B_1$ must be $M_4.$ 
\end{enumerate}
Hence
\begin{equation}
\label{4.1}
B_1\in \{ M_1^{\ell_{1}}M_3M_2^{\ell_{2}}M_4, \; M_1^{\ell_1}M_3M_4, \; M_2^{\ell_2}M_4, \; M_3M_2^{\ell_{2}}M_4, \; M_3 M_4, \; M_4: \ell_1,\ell_2 \ge 1 \} .
\end{equation}

\vspace{3mm}

\noindent For $B_2, B_3,\cdots ,B_{r-1}$, since the last matrix is $M_4$ in the previous block, then the first matrix in these blocks must be either $M_1$ or $M_3$ by requirement (d) above and the last matrix is $M_4$. Hence
\begin{equation}
\label{4.2}
B_k \in \{ M_1^{\ell_{1}}M_3M_2^{\ell_{2}}M_4, \; M_1^{\ell_1}M_3M_4, \; M_3M_2^{l_2}M_4, \; M_3M_4: \ell_1, \ell_2 \ge 1 \}
\end{equation}
for $2 \le k \le r-1$.

\vspace{3mm}

\noindent For the last block $B_r$, it is still starting from $M_1$ or $M_3$ but it is not necessary ending by $M_4$, so
\begin{multline}
\label{4.3}
B_r \in  \{ M_1^{\ell_1}, \; M_1^{\ell_1}M_3,  \; M_1^{\ell_1}M_3M_4, \; M_1^{\ell_{1}}M_3M_2^{\ell_{2}},\\ M_1^{\ell_{1}}M_3M_2^{\ell_{2}}M_4, \; M_3M_2^{\ell_2}, \; M_3M_2^{\ell_2}M_4, \; M_3M_4 :  \ell_1, \ell_2 \ge 1 \}.
\end{multline}

If there is only one block, i.e.,  $r=1$, then either there is no $M_4$ in $M$ or there is only one $M_4$ which is the last matrix in $M$. Then 
\begin{multline}
\label{4.4}
M \in  \{ M_1^{\ell_1}, \; M_1^{\ell_1}M_3,  \; M_1^{\ell_1}M_3M_4, \; M_1^{\ell_{1}}M_3M_2^{\ell_{2}}, \; M_1^{\ell_{1}}M_3M_2^{\ell_{2}}M_4, \\
M_2^{\ell_2}, \; M_2^{\ell_2}M_4, \; 
M_3M_2^{\ell_2}, \; M_3M_2^{\ell_2}M_4, \; M_3M_4 , \; M_4:  \ell_1, \ell_2 \ge 1 \}.
\end{multline}

As  a summary, we have just proved that for even $n$, if $M\in T_n$, then  $M=B_1B_2\cdots B_r$ with $B_1$  in \eqref{4.1}, $B_k$  in \eqref{4.2} for $2 \le k \le r-1$ and $B_r$ in \eqref{4.3} or $M$ is either one of \eqref{4.4}. 

For $L \ge 0$, we define the set
\begin{multline*}
\F := \{  B^{k_1}M_1^{\ell_1}B^{k_2}M_1^{\ell_2}\cdots B^{k_L}M_1^{\ell_L}B^{k_{L+1}}:\\  k_j \ge 1, (2 \le j \le L),  \ell_j \ge 1, (1 \le j \le L), k_1, k_{L+1} \ge 0 \} .
\end{multline*}
For $L=0$, we understand that $\prod_{j=1}^LM_1^{\ell_j}B^{k_{j+1}}=1$ so that $\F$ contains $B^{k_1}$ for $k_1 \ge 0$. 
The following lemma gives a nice factorization of $M$ in $T_n$.
\begin{lem} 
\label{lem 2}
For even $n$, every $M \in T_n$  can be written in the form of 
\begin{equation}
M= T^{\delta_1}B^{k_1}M_1^{\ell_1}B^{k_2}M_1^{\ell_2}\cdots B^{k_L}M_1^{\ell_L}B^{k_{L+1}}T^{\delta_2 } ,   \label{*}
\end{equation}
for $\ell_1, \ell_2, \ldots ,\ell_L  \ge 1, \; k_2, k_3, \cdots ,k_L \ge 1, \; ,  k_1, k_{L+1} \ge 0 $ and $ \delta_1, \delta_2\in \{ 0, 1\} .$
Also, every $M \in T_n$  can be written in the form of 
\begin{equation}
M= T^{\delta_1}B^{k_1}M_1^{\ell_1}B^{k_2}M_1^{\ell_2}\cdots B^{k_L}M_1^{\ell_L}B^{k_{L+1}}T^{\delta_2 } \label{**}
\end{equation}
for 
$\ell_1, \ell_2, \ldots ,\ell_L  \ge 1, \; k_2, k_3, \cdots ,k_L \in \{ 1,2 , 3\} , \; ,  k_1, k_{L+1} \in \{  0, 1, 2, 3\}  $ and $  \delta_1, \delta_2\in \{ 0, 1\} .$
\end{lem}
\begin{proof}
Recall that if $M\in T_n$ for even $n$, then  $M=B_1B_2\cdots B_r$ with $B_1$  in \eqref{4.1}, $B_k$  in \eqref{4.2} for $2 \le k \le r-1$ and $B_r$ in \eqref{4.3} or $M$ is either one of \eqref{4.4}. 

We now consider the case $M=B_1B_2\cdots B_r$. We study the terms in $B_1$  in \eqref{4.1} first. For example, the first the term  in the list  of \eqref{4.1} is
\begin{multline}
M_1^{\ell_1}M_3M_2^{\ell_{2}}M_4=M_1^{\ell_1}BT\underbrace{(TA)(TA)\cdots (TA)}_{\ell_2 \mbox{ times}}TB \\
= M_1^{\ell_1}B\underbrace{(AT)(AT)\cdots (AT)}_{\ell_2 \mbox{ times}}B=M_1^{\ell_1}BM_1^{\ell_{2}}B
\end{multline}
 by \eqref{30}.  Similarly for the other terms in \eqref{4.1}, we  have
 \[
 M_1^{\ell_1}M_3M_4=M_1^{\ell_1}B^2, \;\; M_2^{\ell_2}M_4=TM_1^{\ell_2}B, \;\; M_3M_2^{\ell_2}M_4=BM_1^{\ell_2}B, \;\; M_3M_4=B^2, \;\; M_4=TB.
 \]
 and they all belong to $T^{\delta_1}\F T^{\delta_2}$. Therefore
\[
B_1 \in \{ M_1^{\ell_1}BM_1^{\ell_{2}}B, M_1^{\ell_1}B^2, TM_1^{\ell_2}B, BM_1^{\ell_2}B, B^2, TB \}  \subseteq T^{\delta_1}\F T^{\delta_2}, 
\]
and, for $2 \le k \le r-1$
\[
B_k \in  \{ M_1^{\ell_1}BM_1^{\ell_{2}}B, M_1^{\ell_1}B^2, BM_1^{\ell_2}B, B^2 \} \subseteq T^{\delta_1}\F T^{\delta_2}
\]
because the matrices in the list of \eqref{4.2} are also in \eqref{4.1} where $\ell_1, \ell_2 \ge 1$ and $\delta_1, \delta_2 \in \{ 0,1 \} $. We now consider the last block $B_r$. The matrices in \eqref{4.3} not in \eqref{4.1} are 
\[
M_1^{\ell_1}, \;\; M_1^{\ell_1}M_3=M_1^{\ell_1}BT, \;\; M_1^{\ell_1}M_3M_2^{\ell_2}=M_1^{\ell_1}BM_1^{\ell_2}T, \;\; M_3M_2^{\ell_2}=BM_1^{\ell_2}T.
\]
and therefore
\[
B_r \in \{ M_1^{\ell_1}, M_1^{\ell_1}BT, M_1^{\ell_1}B^2, M_1^{\ell_1}BM_1^{\ell_{2}}T,  M_1^{\ell_1}BM_1^{\ell_{2}}B, BM_1^{\ell_2}T, BM_1^{\ell_2}B, B^2 \} \subseteq T^{\delta_1}\F T^{\delta_2}
\]
where $\ell_1, \ell_2 \ge 1$ and $\delta_1, \delta_2 \in \{ 0,1 \} $.

We now consider the remaining case when $M$ is one of \eqref{4.4}.  Then the only matrix  in \eqref{4.4} but not in \eqref{4.1}, \eqref{4.2} and \eqref{4.3}  is $M=M_2^{\ell_2}=TM_1^{\ell_2}T \in T^{\delta_1}\F T^{\delta_2}$ where $\ell_1, \ell_2 \ge 1$ and $\delta_1, \delta_2 \in \{ 0,1 \} $.

Lastly, it is clearly that if $P_1, P_2 \in \F$, then $P_1P_2\in \F$. Hence, We have proved that  every $M\in T_n$ can be written in the form of  
\[
M= T^{\delta_1}B^{k_1}M_1^{\ell_1}B^{k_2}M_1^{\ell_2}\cdots B^{k_L}M_1^{\ell_L}B^{k_{L+1}}T^{\delta_2 } 
\]
for $k_j \ge 1, (2 \le k \le L),  \ell_j \ge 1, (1 \le j \le L), k_1, k_{L+1} \ge 0$, and $\delta_1, \delta_2=0, 1$. Also $\sum k_j +\sum \ell_j = n$. 
This proves the first assertion of the lemma. 

For the second assertion, we observe that $B^4=B^2$ and $B^5=B^3$. This completes the proof of the lemma.
\end{proof}

Note that in \eqref{*}  we have
$$\sum_{j=1}^L (k_j +\ell_j) +k_{L+1}= n.$$
and in \eqref{**}, we have
$$\sum_{j=1}^L (k_j +\ell_j)+k_{L+1} \le n$$
because $k_j \in \{  1, 2, 3 \}$ for $2 \le j \le L$.

The characteristic polynomial of $M_1^2=AD$ is $g(x)=x^3-5x^2+12x-16$.  Then $g(x)=0$ has one real root $\lambda$ and two complex roots $\lambda'$ and $\hb{\lambda'}$  with $|\lambda | > |\lambda'|$ and
\[
\lambda =\frac{(71+6\sqrt{177})^{1/3}- (-71+6\sqrt{177})^{1/3}+5}{3} =2.75217177\cdots
\]
and
\begin{eqnarray*}
\lambda' &=& -\frac{(71+6\sqrt{177})^{1/3}-(-71+6\sqrt{177})^{1/3}-10}{6}  \\
& & \qquad +i\sqrt{3}\left( \frac{(71+6\sqrt{177})^{1/3}+ (-71+6\sqrt{177})^{1/3}}{6}\right) \\ 
&=&1.12391411\cdots +2.13316845\cdots i.
\end{eqnarray*}
Since these eigenvalues are distinct, there is a nonsingular  matrix $S$ such that $S^{-1}M_1^2S=\Lambda$ with $\Lambda
=\mbox{diag}[\lambda , \lambda' ,\hb{\lambda'}]$, the diagonal matrix with diagonal elements $\lambda , \lambda'$ and $\hb{\lambda'}$ respectively. Since 
\begin{equation}
M_1^{2k }=S \Lambda^{k} S^{-1}
\label{19}
\end{equation}
where
\[
S:=
\begin{pmatrix}
s_1 & s_3 &  \hb{s_3} \\ s_2  & s_4 & \hb{s_4}\\
1 & 1 & 1
\end{pmatrix} ,
\] 
\[
s_1:=\frac{1}{66}\left( (10-\sqrt{177})(71+6\sqrt{177})^{1/3}  - (10+\sqrt{177})(-71+6\sqrt{177})^{1/3}   -22\right) ,
\]
\[
s_2:=\frac{1}{66} \left(\left( 1+\sqrt {177} \right) ({71+6\,\sqrt {177}})^{1/3} - \left( 1-\sqrt {177} \right) ({-71+6\,\sqrt {177}})^{1/3}+44\right) ,
\]
\begin{multline*}
s_3:=-\frac{1}{132} \left(  \left( 1-i
\sqrt {3} \right)  \left( 10- \sqrt {177} \right) (71+6\,
\sqrt {177})^{1/3} \right.\\ \left. - \left( 1+i\sqrt {3} \right) \left( 10+\sqrt {177} \right)  (-71+6\sqrt {177})^{1/3} +44 \right)
\end{multline*}
and
\begin{multline*}
s_4:=\frac {1}{132}\left( -  \left( 1-i\sqrt {3} \right) \left( 1+\sqrt {177} \right) 
(71+6\,\sqrt {177})^{1/3}\right.\\ \left.
+ \left( 1+i\sqrt {3}
 \right)  \left( 1-\sqrt {177} \right) (-71+6\,\sqrt {177})^{1/3}+88 \right) .
\end{multline*}

For any matrix $M$, let $\| M \|$ be the spectral norm of $M$ induced by the $L_2$ norm. That is if $M$ is an $n\times n$ matrix, $M^{\ast}$ is its conjugate transpose and $\lambda_1, \lambda_2, \ldots , \lambda_n$ are the eigenvalues of $M^{\ast} M$ (the eigenvalues are all real), then $\| M \| = \sqrt{\max \{ \lambda_1, \lambda_2, \ldots , \lambda_n \}}$. 
It is known that the spectral radius of $M$ is less than $\| M \|$. We should remark that spectral norm satisfies the submultiplicativity, i.e,  $\| AB \| \le \| A\| \| B \|$ but spectral radius does not.

For any positive integer $k$, if $\ell$ is even, then we have
 \[
 \| M_1^\ell\| = \| S \Lambda^{\ell /2} S^{-1}  \| \le \| S \| \| \Lambda^{\ell/2} \| \| S^{-1} \|  \ll    \lambda^{\ell/2} 
 \]
because  $\| \Lambda^{\ell /2}  \| = \lambda^{\ell /2}$.
If $\ell$ is odd, then we have
 \begin{eqnarray*}
 \| M_1^\ell\| &=&\|M_1 M_1^{\ell-1}\|=\| M_1S \Lambda^{(\ell-1)/2} S^{-1}  \| \\
 & \le & \| M_1\| \| S \| \| \Lambda^{(\ell-1)/2} \| \| S^{-1} \|  \ll  \lambda^{(\ell-1)/2} \ll  \lambda^{\ell/2}.
 \end{eqnarray*}
 Therefore, for any $\ell \ge 1$, we have 
 \begin{equation}
 \label{20}
 \| M_1^\ell \| \ll \lambda^{\ell/2}.
  \end{equation}
Since the implicit constant in \eqref{20} may not be less than $1$, so we cannot apply this estimate directly to \eqref{**} to prove that $\| M \| \ll \lambda^{n/2}$ for any $M\in T_n$ because the implicit constants may be multiplied together to be large when $L$ is in the order of $n$.  As a result, we need to make use of the factorization in \eqref{**} when we  estimate $\| M \|$. Among all the terms $M_1^{\ell}B^{k}$, the term $M_1B$ is the worst because $\| M_1B \| > \lambda$ and needs special consideration (see Lemma 4 below).
  
The following lemma deals with the product of $2L$ matrices without $M_1B$ factor in \eqref{*} or  \eqref{**}.
\begin{lem} 
\label{lem 3}
We have
\begin{equation}
\label{21}
\| (M_1^{\ell_1}B^{k_1})(M_1^{\ell_2}B^{k_2}) \cdots (M_1^{\ell_{L-1}}B^{k_{L-1}}M_1^{\ell_L}B^{k_{L}}) \| \le ((1.0000005)^2\lambda)^{\frac12 \sum_{j=1}^L(k_j+\ell_j)}
\end{equation}
for any $(k_j, \ell_j)$ with $k_j \in \{ 1, 2, 3 \}$, $\ell_j \geq 1$ 
and $(\ell_j,k_j)\neq (1,1)$ for $1\le j \le L$. 
\end{lem}
\begin{proof}
We will prove \eqref{21} by induction on $L$. 

We first claim that 
\begin{equation}
\label{22}
\| M_1^{\ell}B^k\| \le ((1.0000005)^2\lambda)^{\frac12 (\ell +k)}
\end{equation}
for all $k \in \{1, 2, 3\} , \ell \geq 1$ and $(\ell ,k)\neq (1,1)$.

If $\ell$ is even, then
\begin{multline*}
\| M_1^\ell B^k \| = \| (M_1^2)^{\ell/2}B^k \| =\| S \Lambda^{\ell /2} S^{-1} B^k\| \\
\le \| S \|\cdot \| S^{-1} B^k\| \lambda^{\ell /2} =\frac{\| S\| \| S^{-1} B^k\| }{\lambda^{k/2}}\lambda^{\frac12 (\ell+k)}=c_k \lambda^{\frac12 (\ell+k)}
\end{multline*}
where $c_k = \frac{\| S\| \| S^{-1} B^k\| }{\lambda^{k/2}}$. If 
$\ell \ge \frac{\log c_k}{\log 1.0000005} -k$, then 
\[
c_k \le 1.0000005^{\ell +k}.
\] 
It then follows that 
\begin{equation}
\label{***}
\| M_1^\ell B^k \| \le c_k \lambda^{\frac12 (\ell+k)} \le ((1.0000005)^2\lambda)^{\frac12 (\ell +k)}.
\end{equation}
It remains to check that \eqref{22} holds for $\ell \le \frac{\log c_k}{\log 1.0000005} -k$. Using
 {\tt Maple}, we can compute $\| S\|$ and $ \| S^{-1} B^k\|$ and find out that 
\[
\frac{\log c_k}{\log 1.0000005} -k = 
\begin{cases}
1318902.018 & \mbox{ when } k=1, \\
991945.7928 &  \mbox{ when } k=2,\\
-20445.79861 & \mbox{ when } k=3.
\end{cases}
\]
We then use {\tt Maple} to check \eqref{22} holds for $\ell \le 1318902$ when $k=1$ and 
$\ell \le 991945$ when $k=2$ and do not need to check for $k=3$. 

Similarly, if $\ell$ is odd, then
\begin{multline*}
\| M_1^\ell B^k \| = \| (M_1^2)^{(\ell-1)/2}M_1B^k \| =\| S \Lambda^{(\ell -1)/2} S^{-1} M_1B^k\| \\ 
\le \| S\| \| S^{-1}M_1 B^k\|\lambda^{\frac12 (\ell -1)}= \frac{\| S\| \| S^{-1}M_1 B^k\| }{\lambda^{(k+1)/2}}\lambda^{\frac12 (\ell+k)}=c_k' \lambda^{\frac12 (\ell+k)}
\end{multline*}
where $c_k' = \frac{\| S\| \| S^{-1} M_1B^k\| }{\lambda^{(k+1)/2}}$. If 
$\ell \ge \frac{\log c_k'}{\log 1.0000005} -k$, then 
\[
\| M_1^\ell B^k \| \le c_k' \lambda^{\frac12 (\ell+k)} \le ((1.0000005)^2\lambda)^{\frac12 (\ell +k)}.
\]
It remains to check that \eqref{22} holds for $\ell \le \frac{\log c_k'}{\log 1.0000005} -k$. Using 
{\tt Maple} again, we have
\[
\frac{\log c_k'}{\log 1.0000005} -k = 
\begin{cases}
1187950.952 & \mbox{ when } k=1, \\
862238.8518  &  \mbox{ when } k=2, \\
-150152.7391 & \mbox{ when } k=3.
\end{cases}
\]
We do not need to check this again because these bounds are smaller than the bounds in the
case where $\ell$ is even.  This proves our claim. Hence \eqref{21} is true for $L=1$. 

We now assume that \eqref{21} is true for $L-1$. Then we have
\begin{eqnarray*}
&    & \| M_1^{\ell_1}B^{k_1}M_1^{\ell_2}B^{k_2} \cdots M_1^{\ell_L}B^{k_{L}} \|  \\
& \le & \| M_1^{\ell_1}B^{k_1} \| \cdot \| M_1^{\ell_2}B^{k_2} \cdots M_1^{\ell_L}B^{k_{L}} \|  \\
& \le & ((1.0000005)^2\lambda)^{\frac12 (\ell_1 +k_1)} ((1.0000005)^2\lambda)^{\frac12 \sum_{j=2}^{L}(\ell_j +k_j)} \\
&  = & ((1.0000005)^2\lambda)^{\frac12 \sum_{j=1}^{L}(\ell_j +k_j)}.
\end{eqnarray*}
This proves our lemma.
\end{proof}
The calculations in the proof of Lemma 3 are simple in using  {\tt Maple} and they have been verified independently a couple of times. 

The term $M_1B$ is bad because $\| M_1B \| > (1.0000005)^2\lambda$ which does not satisfy \eqref{21}. For handling this term, we have the following lemma.

\begin{lem}
\label{lem 4} 
We have 
\begin{equation}
\| (M_1B)^{\ell} \| \le  ((1.0000005)^2\lambda )^\ell
\label{*.7}
\end{equation}
for any $\ell \ge 2$ and 
\begin{equation}
\label{2.1}
\| M_1BM_1^\ell B^k \| \le ((1.0000005)^2\lambda)^{(\ell+k+2)/2}
\end{equation}
for $\ell \ge 1$ and $0 \le k \le 3$. Note that $\| M_1B \| = 2\sqrt{2} > (1.0000005)^2\lambda$ so that \eqref{*.7} does not hold for $\ell =1$.
\end{lem}
\begin{proof}
First we write, $M_1B=S_1\Lambda_1S_1^{-1}$ where $\Lambda_1
=\mbox{diag}[\frac{1+\sqrt{17}}{2}, \frac{1-\sqrt{17}}{2}, 0]$ and 
 \[
 S_1=\begin{pmatrix}
\frac{-5+\sqrt{17}}{4}  & \frac{-5-\sqrt{17}}{4}  &  1 \\ \frac{-3+\sqrt{17}}{2}   & \frac{-3-\sqrt{17}}{2}  & 0\\
1 & 1 & 0
\end{pmatrix} .
 \]
 Hence,  we have
 \begin{eqnarray}
  \| (M_1B)^{\ell} \| &= & \| S_1 \Lambda_1^{\ell} S_1^{-1} \| \le \| S_1\| \cdot \| S_1^{-1} \| \|  \Lambda_1^{\ell} \| \nonumber  \\
 & = & (5.61541131\cdots ) \left( \frac{1+\sqrt{17}}{2}\right)^\ell \nonumber  \\
 & \le & ((1.0000005)^2\lambda )^\ell  \nonumber 
 \end{eqnarray}
 for $\ell \ge 25$ because $\| S_1\|=4.38008933\cdots , \| S_1^{-1} \| = 1.282031231$ and $\|  \Lambda_1^{\ell} \|  = \left( \frac{1+\sqrt{17}}{2}\right)^\ell$.  For $2 \le \ell \le 24$, we use {\tt Maple} to check directly that \eqref{*.7} is true. This proves \eqref{*.7}.

For $\ell \ge 6$, we have
\begin{eqnarray*}
\| M_1BM_1^\ell B^k \| & =& \| M_1BM_1^4 M_1^{\ell -4} B^k \| \\
& \le & \| M_1BM_1^4\| \cdot \| M_1^{\ell -4} B^k \| \\
& \le & ((1.0000005)^2\lambda)^{3}((1.0000005)^2\lambda)^{(\ell -4+k)/2} \\
& \le & ((1.0000005)^2\lambda)^{(\ell +k +2)/2} 
\end{eqnarray*}
by Lemma \ref{lem 3} with $(\ell -4,k)\neq (1,1)$ and 
\[
\| M_1BM_1^4\|=19.97828015\cdots <  20.84624870=((1.0000005)^2\lambda)^{3}.
\]
For $1\le \ell \le 5$ and $1 \le k \le 3$, we use {\tt Maple} to check that \eqref{2.1} is true for these 
$\ell$ and $k$. Among these 15 terms, we have
\[
\frac{\| M_1BM_1^\ell B^k \|}{((1.0000005)^2\lambda)^{(\ell+k+2)/2}} \le 0.92894011 < 1.
\]
This proves \eqref{2.1}.
\end{proof}


\begin{thm} 
\label{thm 2}
We have 
\begin{equation}
\label{2.2}
\| M \| \ll ((1.0000005)^2\lambda)^{n/2} 
\end{equation}
 for any $M \in T_n$. 
\end{thm}
\begin{proof}
If $n$ is even , for any $M\in T_n$, by Lemma \ref{lem 2},
\[
M= T^{\delta_1}B^{k_1}M_1^{\ell_1}B^{k_2}M_1^{\ell_2}\cdots B^{k_L}M_1^{\ell_L}B^{k_{L+1}}T^{\delta_2 } 
\]
for $\ell_1, \ell_2, \ldots ,\ell_L  \ge 1, \; k_2, k_3, \cdots ,k_L \in \{ 1,2 , 3\} , \; ,  k_1, k_{L+1} \in \{  0, 1, 2, 3\}  $ and $  \delta_1, \delta_2\in \{ 0, 1\} .$

\vspace{3mm}

We divide our consideration into two cases: (1) $k_{L+1}=1,2,3$ and (2) $k_{L+1}=0$.

\vspace{3mm}

\noindent Case (1): If $k_{L+1}=1,2,3$, then
\begin{eqnarray}
\| M \| & \le &  \left( \| T^{\delta_1}B^{k_1} \| \cdot \| T^{\delta_2 } \|\right) \| M_1^{\ell_1}B^{k_2}M_1^{\ell_2}\cdots B^{k_L}M_1^{\ell_L}B^{k_{L+1}} \|  \nonumber \\
& \ll & \| M_1^{\ell_1}B^{k_2}M_1^{\ell_2}\cdots B^{k_L}M_1^{\ell_L}B^{k_{L+1}} \|.  \label{*.1}
\end{eqnarray}
We further divide the consideration into $4$ cases depending on how $M$ contains the term $M_1B$ as a factor.
\begin{enumerate}
\item[(i)] If there is no $(\ell_j, k_{j+1})=(1,1)$, then by Lemma \ref{lem 3}, we have 
\begin{equation}
\label{*.2}
\| M_1^{\ell_1}B^{k_2}M_1^{\ell_2}\cdots B^{k_L}M_1^{\ell_L}B^{k_{L+1}} \| \le ((1.0000005)^2\lambda)^{n/2}.
\end{equation}
\item[(ii)] If there is $(\ell_j, k_{j+1})=(1,1)$ for $1\le k \le L-1$, then by grouping as a product of $(M_1^{\ell_j}B^{k_{j+1}}M_1^{\ell_{j+1}}B^{k_{j+2}})=(M_1BM_1^{\ell_{j+1}}B^{k_{j+2}})$ and possibly $(M_1^{\ell_L}B^{k_{L+1}})$ for $L$ is
 odd and using \eqref{21}, \eqref{***} and  \eqref{2.1}, we have 
\begin{equation}
\label{*.3}
\| M_1^{\ell_1}B^{k_2}M_1^{\ell_2}\cdots B^{k_L}M_1^{\ell_L}B^{k_{L+1}} \|  \le ((1.0000005)^2\lambda)^{n/2}.
\end{equation}
\item[(iii)] If $(\ell_L, k_{L+1})=(\ell_{L-1}, k_{L}) = \cdots = (\ell_{t+1}, k_{t+2})  =(1,1)$ but $(\ell_{t}, k_{t+1}) \neq (1,1)$,  for some $1 \le t \le L-2$ then 
\begin{eqnarray*}
&   &  \| M_1^{\ell_1}B^{k_2}M_1^{\ell_2}\cdots B^{k_L}M_1^{\ell_L}B^{k_{L+1}} \|  \\   & \le  &\| M_1^{\ell_1}B^{k_2}\cdots M_1^{\ell_{t}}B^{k_{t+1}}\| \cdot \| (M_1B)^{L-t} \|  \\
& \le &((1.0000005)^2\lambda)^{\sum_{j =1}^t (\ell_j + k_{j+1})/2} \| (M_1B)^{L-t} \|    \\
\end{eqnarray*}
by the same argument as in (i) or (ii).  In view of  \eqref{*.7}, since $L-t \ge 2$, we have $\| (M_1B)^{L-t} \| \le  ((1.0000005)^2\lambda)^{L-t}$ and hence
\begin{eqnarray}
&    & \| M_1^{\ell_1}B^{k_2}M_1^{\ell_2}\cdots B^{k_L}M_1^{\ell_L}B^{k_{L+1}} \|  \nonumber \\ 
 & \le &  ((1.0000005)^2\lambda)^{\sum_{j =1}^t (\ell_j + k_{j+1})/2 + (L-t)}  \nonumber  \\
& = & ((1.0000005)^2\lambda)^{\sum_{j =1}^L (\ell_j + k_{j+1})/2}  \le   ((1.0000005)^2\lambda)^{n/2}\label{*.4}
\end{eqnarray}
\item[(iv)] if $(\ell_L, k_{L+1})=(1,1)$ but $(\ell_{L-1}, k_{L}) \neq (1,1)$, then 
\begin{eqnarray}
&      & \| M_1^{\ell_1}B^{k_2}M_1^{\ell_2}\cdots B^{k_L}M_1^{\ell_L}B^{k_{L+1}} \|  \nonumber \\ 
& \le & \| M_1^{\ell_1}B^{k_2}M_1^{\ell_2}\cdots B^{k_L}\| \| M_1B \|   \nonumber \\ 
& \ll &   ((1.0000005)^2\lambda)^{n/2-1} \ll ((1.0000005)^2\lambda)^{n/2}\label{*.5}
\end{eqnarray}
\end{enumerate}
by (i), (ii) and (iii). Hence in view of \eqref{*.1} - \eqref{*.5}, we have
\[
\| M \| \ll ((1.0000005)^2\lambda)^{n/2}.
\]

\vspace{3mm}

\noindent Case (2): If $k_{L+1}=0$, then 
\begin{eqnarray}
\| M \| & \ll & \| M_1^{\ell_1}B^{k_2}M_1^{\ell_2}\cdots B^{k_L}M_1^{\ell_L} \|  \nonumber  \\
& \le & \| M_1^{\ell_1}B^{k_2}M_1^{\ell_2}\cdots B^{k_L}\| \cdot \| M_1^{\ell_L} \|  \nonumber \\
& \ll & \lambda^{\ell_L/2} \| M_1^{\ell_1}B^{k_2}M_1^{\ell_2}\cdots B^{k_L}\| \nonumber \\
& \ll &   ((1.0000005)^2\lambda)^{n/2}  \nonumber 
\end{eqnarray}
by \eqref{20} and the term $\| M_1^{\ell_1}B^{k_2}M_1^{\ell_2}\cdots B^{k_L}\|$ can be treated 
in the same way as Case (1). Hence we get \eqref{2.2} for $n$ is even. 

If $n$ is odd, then $M = M' \M$ with $M'\in T_{n-1}$ and $\M \in T_1$. So
\begin{eqnarray*}
\| M \| & \le &  \| M'\| \cdot \| \M\| \ll \| M' \| \\
&  \ll &   ((1.0000005)^2\lambda)^{(n-1)/2} 
\end{eqnarray*}
by applying \eqref{2.2} for $M'$ as $n-1$ is even.
This proves \eqref{2.2} for $n$ is odd. 

As a result, we have
\[
|a_{k_n}|, |b_{k_n'}| \le \sqrt{|a_{k_n}|^2+2|b_{k_n'}|^2} = \| \omega_n(k_n) \| \le \| M \| \| \omega_1(k_1)\| \ll  ((1.0000005)^2\lambda)^{n/2}.
\]
This completes the proof of Theorem \ref{thm 2}.
\end{proof}

In view of Theorem \ref{thm 2}, we prove the upper bound in Theorem \ref{thm 1}.

We remark that the upper bound and lower bound agree with 6 decimal digits of the constants in the exponents. In order to improve our upper bound to agree with 7 decimal digits, we need to verify 
Lemma 3 for $\ell \le 3.297256652\cdots 10^7$ for $k=1$ and $\ell \le 2.479868687\cdots 10^7$  with $(1.00000002)$ in the place of $(1.000005)$.

\section{Lower bound for the autocorrelation coefficients}

This is in fact a result in \cite{Ta2} but its proof contains some mistakes and typos. For example, $k_n$ should be $\frac13 (2L_n-1)$ instead of $\frac13 (L_n+1)$ as stated in \cite{Ta2} when $n$ is odd. 
Also the coefficients of (6) and (7) in \cite{Ta2} are incorrect (cf. \eqref{A} and \eqref{B} below). 

In view of \eqref{1} and \eqref{3},  for $|z|=1$ we have
\begin{eqnarray*}
          \left| P_n(z) \right|^2 
& = &  z^{2^{n-1}} \hb{P_{n-1}(z)}Q_{n-1}(z) + \hb{z}^{2^{n-1}}P_{n-1}(z) \hb{Q_{n-1}(z)} + 2^n.
\end{eqnarray*}
Using \eqref{3} again, we have
\begin{eqnarray*}
          \left| P_n(z) \right|^2 
& = &  2\left( z^{L_{n-1}} + \hb{z}^{L_{n-1}}\right)|P_{n-2}(z)|^2  -  \left( z^{3L_{n-2}} - \hb{z}^{L_{n-2}}\right)\hb{P_{n-2}(z)}Q_{n-2}(z) \\
&    & +   \left( z^{L_{n-2}} - \hb{z}^{3L_{n-2}}\right)P_{n-2}(z)\hb{Q_{n-2}(z)}- L_{n-1}\left( z^{L_{n-1}} + \hb{z}^{L_{n-1}}\right) + 2^n.
\end{eqnarray*}
Let 
\begin{equation}
k_n=\frac13 (2L_n +(-1)^n) \quad \mbox{ and } \quad k_n':=k_n-L_n.
\label{3.1}
\end{equation}
Thus, for $k_n \in S_n$, we have
\begin{eqnarray}
&     &     \h{\left| P_n(z) \right|^2}(k_n) \nonumber \\
& = &  2\h{\left( z^{L_{n-1}} + \hb{z}^{L_{n-1}}\right)|P_{n-2}(z)|^2}(k_n)  -  \h{\left( z^{3L_{n-2}} - \hb{z}^{L_{n-2}}\right)\hb{P_{n-2}(z)}Q_{n-2}(z) }(k_n) \nonumber \\
&    & +  \h{ \left( z^{L_{n-2}} - \hb{z}^{3L_{n-2}}\right)P_{n-2}(z)\hb{Q_{n-2}(z)}- L_{n-1}\left( z^{L_{n-1}} +\hb{z}^{L_{n-1}}\right) } (k_n) \nonumber \\
& = & 2\h{|P_{n-2}(z)|^2}(k_n-L_{n-1}) + 2\h{|P_{n-2}(z)|^2} (k_n+L_{n-1}) \nonumber \\
&    & -\h{\hb{P_{n-2}(z)}Q_{n-2}(z)}(k_n-3L_{n-2}) + \h{\hb{P_{n-2}(z)}Q_{n-2}(z)}(k_n+L_{n-2}) \nonumber \\ 
&   & + \h{P_{n-2}(z)\hb{Q_{n-2}(z)}}(k_n-L_{n-2})-\h{P_{n-2}(z)\hb{Q_{n-2}(z)}}(k_n+2L_{n-2})  \nonumber \\ 
&    & -L_{n-1}\h{1}(k_n-L_{n-1}) +L_{n-1}\h{1}(k_n+L_{n-1}).  \nonumber 
\end{eqnarray}
In view of \eqref{3.1}, we have $k_n+L_{n-1}, k_n+L_{n-2}, k_n-L_{n-2}, k_n+2L_{n-2} \not\in S_{n-2}$, $k_n \neq \pm L_{n-1}$  so that their corresponding Fourier coefficients are $0$ and
\[
k_n-3L_{n-2}=k_{n-2}-L_{n-2}= k_{n-2}'.
\]
It follows that 
\begin{equation}
     \h{\left| P_n(z) \right|^2}(k_n) = 2\h{|P_{n-2}|^2}(k_{n-2})  -\h{\hb{P_{n-2}}Q_{n-2}}(k_{n-2}').   
\end{equation}
Similarly, we have 
\begin{eqnarray*}
&    &  \hb{P_n(z)}Q_{n}(z)  \\
& = &  2\left( \hb{z}^{L_{n-1}}-z^{L_{n-1}}  \right)|P_{n-2}(z)|^2  +  \left( z^{3L_{n-2}} +2z^{L_{n-2}}+\hb{z}^{L_{n-2}}\right)\hb{P_{n-2}(z)}Q_{n-2}(z) \\
&    & -   \left( \hb{z}^{3L_{n-2}} -2\hb{z}^{L_{n-2}}+z^{L_{n-2}}\right)P_{n-2}(z)\hb{Q_{n-2}(z)}+L_{n-1}\left( z^{2L_{n-2}} - \hb{z}^{2L_{n-2}}\right) 
\end{eqnarray*}
and hence
\begin{eqnarray}
& & \h{\hb{P_n(z)}Q_{n}}(k_n')= 2\h{|P_{n-2}|^2}(k_{n-2}) \nonumber \\ & &  \qquad \qquad \qquad +\h{\hb{P_{n-2}}Q_{n-2}}(k_{n-2}') + 2 \h{P_{n-2}\hb{Q_{n-2}}}(k_{n-2}')\label{A}
\end{eqnarray}
and
\begin{eqnarray*}
 &  &    P_n(z)\hb{Q_{n}}(z)   =   2\left( {z}^{L_{n-1}}-\hb{z}^{L_{n-1}}  \right)|P_{n-2}(z)|^2  -   \left( {z}^{3L_{n-2}} -2{z}^{L_{n-2}}+\hb{z}^{L_{n-2}}\right)  \\
&    & \qquad + \left( \hb{z}^{3L_{n-2}} +2\hb{z}^{L_{n-2}}+{z}^{L_{n-2}}\right)P_{n-2}(z)\hb{Q_{n-2}(z)}-L_{n-1}\left( z^{2L_{n-2}} - \hb{z}^{2L_{n-2}}\right) 
\end{eqnarray*}
and hence
\begin{eqnarray}
\label{B}
& & \h{P_n(z)\hb{Q_{n}}}(k_n')=-2\h{|P_{n-2}|^2}(k_{n-2}) \nonumber \\ & & \qquad \qquad \qquad-\h{\hb{P_{n-2}}Q_{n-2}}(k_{n-2}') + 2 \h{P_{n-2}\hb{Q_{n-2}}}(k_{n-2}').
\end{eqnarray}
Therefore, we have
\[
\w_n =M_1 \w_{n-2}
\]
for $n \ge 2$ where $M_2=\begin{pmatrix}
2 & -1 & 0 \\
2 & 1 & 2 \\
-2 & -1 & 2
\end{pmatrix}$. Also we have $\Lambda_2 = {S'}^{-1}M_2S'$ where $\Lambda$ is $3\times 3$ diagonal matrix with $\lambda, (\lambda'), (\overline{\lambda'})$ entries.
where
\[
S':=
\begin{pmatrix}
s_5 & s_6 & \hb{s_6} \\
s_7 & s_8 & \hb{s_8} \\
1 & 1 & 1 
\end{pmatrix}
\]
and
\begin{eqnarray*}
s_5&:=&{\frac { -11\left( 1+ \,\sqrt {177} \right)\left( 71+6\,\sqrt {177
}\right)^{1/3}-\left( 103-7\,\sqrt {177} \right)  \left( 71+
6\,\sqrt {177} \right) ^{2/3}+242}{1452}}, \\
s_6&:=&{\frac { 11 \left( 1-\,i\sqrt {3} \right)  \left( \sqrt {177}+1
 \right) \left( 71+6\,\sqrt {177
}\right)^{1/3}+\left( 1+i\sqrt {3} \right)  \left( 103-7\,\sqrt {177} \right) 
 \left( 71+6\,\sqrt {177} \right) ^{2/3}+484}{2904}}, \\
 s_7&:=&
{\frac { 11\left( -21+\,\sqrt {177} \right) 
\left( 71+6\,\sqrt {177
}\right)^{1/3}+ \left(-39 + 5\,\sqrt {177}
 \right) \left( 71+6\,\sqrt {177} \right) ^{2/3} }{726}}, \\
s_8&:=&
{\frac { 11 \left( 1-\,i\sqrt {3} \right)  \left( 21 -\sqrt {177}
 \right) \left( 71+6\,\sqrt {177
}\right)^{1/3} +\left( 1+  i\sqrt {
3} \right)  \left( 39-5\,\sqrt {177} \right)  \left( 71+6\,\sqrt {
177} \right) ^{2/3}}{1452}}.
\end{eqnarray*}

If $n$ is even, then inductively,  we have
\[
\w_n = M_2^{n/2}\w_0
\]
where $\w_0=\begin{pmatrix} 0 \\ 1 \\ 1 \end{pmatrix}$ because $|P_0|^2(z)=(\hb{P_0}Q_0)(z)=(P_0\hb{Q_0})(z)=1$ and $k_0=1, k_0'=0$. In view of \eqref{19}, we have 
\[
\w_n=S' \Lambda^{n/2} (S')^{-1}\w_0 
\] 
with $\Lambda
=\mbox{diag}[\lambda , \lambda' ,\hb{\lambda'}]$. 
Therefore, there are constants $a,b,c$ such that 
\begin{eqnarray*}
\left| \h{|P_n|^2}\left( \frac13 (2L_n+1)\right)\right| & = &  \left| a \lambda^{n/2} + b (\lambda')^{n/2} + \hb{b} (\hb{\lambda'})^{n/2} \right|
\\
& \ge & |a|\lambda^{n/2}  - 2|b||\lambda'|^{n/2} \gg \lambda^{n/2} 
\end{eqnarray*}
 where $a=-0.38215952\cdots $ and $b=0.19107976\cdots + i0.88541019\cdots$ and $|\lambda | > | \lambda'|$. This proves
\[
 \max_{1 \leq k \leq L_n-1}{|a_k|} \ge \left| \h{|P_n|^2}\left( \frac13 (2L_n+1)\right) \right| \gg \lambda^{n/2}  .
\]

If $n$ is odd, then inductively,  we have
\[
\w_n = M_2^{(n-1)/2}\w_1
\]
where $\w_1=\begin{pmatrix} 1 \\ 1 \\ -1 \end{pmatrix}$ because 
\[
|P_1|^2(z)=\hb{z}+1+z, \quad (\hb{P_1}Q_1)(z)=\hb{z}-z, \quad (P_1\hb{Q_1})(z)=-\hb{z}+z
\]
 and $k_1=1, k_1'=-1$.  In view of \eqref{19}, we have 
\[
\w_n=S' \Lambda^{(n-1)/2} (S')^{-1}\w_1 
\] 
with $\Lambda
=\mbox{diag}[\lambda , \lambda' ,\hb{\lambda'}]$. 
Therefore, there are constants $a',b',c'$ such that 
\begin{eqnarray*}
\left| \h{|P_n|^2}\left( \frac13 (2L_n-1)\right)\right| & = &  \left| a' \lambda^{(n-1)/2} + b' (\lambda')^{(n-1)/2} + \hb{b'} (\hb{\lambda'})^{(n-1)/2} \right|
\\
& = &  \left| (a'\lambda^{-1/2}) \lambda^{n/2} +  (b'\lambda'^{-1/2}) (\lambda')^{n/2}+  (\hb{b'} \lambda'^{-1/2}) (\hb{\lambda'})^{n/2} \right|
\\
& \ge  & (|a'|\lambda^{-1/2}) \lambda^{n/2} - (2 |b'||\lambda'|^{-1/2})|\lambda'|^{n/2}\\
& \gg & \lambda^{n/2}
\end{eqnarray*}
where $a'=0.28744961\cdots $ and $b'=0.3562751947-i0.3300357859$. 
This proves for all odd $n$, we have
\[
 \max_{1 \leq k \leq L_n-1}{|a_k|} \ge \left| \h{|P_n|^2}\left( \frac13 (2L_n-1)\right) \right| \gg \lambda^{n/2}.
\]
The lower bounds for $ \max_{1 \leq k \leq L_n-1}{|b_k|}$ can be proved in the same way.
This proves  the lower bounds for Theorem \ref{thm 1}.

\section{Application of Theorem \ref{thm 1}}

Let
$$
M_q(f) = \left( \frac{1}{2\pi} \int_{-\pi}^{\pi}|f(\theta)|^q \, d\theta \right)^{1/q}, \qquad q > 0.
$$
Our next lemma is due to Littlewood, see Theorem 1 in \cite{Li1}.

\begin{lem}
\label{lem 4.1}
Let $n \ge 1$. If the trigonometric polynomial $f$ of the form
$$f(t) = \sum_{m=0}^{L_n-1}{b_m \cos(mt + \alpha_m)}\,, \qquad b_m, \alpha_m \in {\mathbb R}\,,$$ 
satisfies
$$M_1(f) \geq c\mu\,, \qquad \mu := M_2(f)\,,$$
where $c > 0$ is a constant, $b_0 = 0$,
$$s_{\lfloor (L_n-1)/h \rfloor} = \sum_{m=1}^{\lfloor (L_n-1)/h \rfloor}{\frac{b_m^2}{\mu^2}} \leq 2^{-9}c^6$$
for some constant $h>0$, and $v \in {\mathbb R}$ satisfies
$$|v| \leq V = 2^{-5}c^3\,,$$
then
$${\mathcal N}(f,v) > Dh^{-1}c^52^n\,,$$
where ${\mathcal N}(f,v)$ denotes the number of real zeros of $f-v\mu$ in $(-\pi,\pi)$, and $D > 0$ is
an absolute constant.
\end{lem}

\vspace{3mm}

\noindent {\bf Proof of Theorem \ref{thm 4}.}

\vspace{3mm}

Let the trigonometric polynomial $f$ be defined by $f(t) := R(t)-L_n$. We show that
$f$ satisfies the assumptions of the Lemma with $c=1/4$ and $h := 2^{(2\lambda - 1)n}$ if $n$ is sufficiently large.
Note that Littlewood in \cite{Li1} evaluated $M_4(P_n)$ and found that $M_4(P_n) \sim (4^{n+1}/3)^{1/4}$. 
Hence $\mu := M_2(f) \sim (1/3)^{1/2}2^n$.
Also, 
it follows from \eqref{2} that $M_{\infty}(f) \leq 2^n$, hence
$$(1/3)2^{2n} \sim (M_2(f))^2 \leq M_1(f)M_{\infty}(f) \leq 2^nM_1(f)$$
implies that $M_1(f) \geq \mu /2$ if $n$ is sufficiently large.
Obviously $f$ is of the form
$$f(t) = \sum_{m=1}^{L_n-1}{b_m \cos(mt)}\,, \qquad b_m \in {\mathbb R}\,,$$
where the assumptions imply that
\begin{eqnarray*} s_{\lfloor (L_n-1)/h \rfloor} & =  & \sum_{m=1}^{\lfloor (L_n-1)/h \rfloor}{\frac{b_m^2}{\mu^2}}  
\leq\left( \frac{2^n-1}{h} \right)\, \left( \frac{4(C2^{(\lambda - \varepsilon )n})^2}{\mu^2} \right) \\
& \leq & \left( \frac{2^n}{2^{(2\lambda - 1)n}}\right) \, \left(\frac{4C^22^{(2\lambda - 2\varepsilon )n}}{2^n/4} \right) \leq 16C^2 2^{-2\varepsilon n} \\
& \leq & 2^{-9}c^6  \end{eqnarray*}
if $n$ is sufficiently large. So $f$ satisfies the assumptions of Lemma 3.6 with $c=1/2$ and
$h:=2^{(2\lambda - 1)n}$ if $n$ is sufficiently large, indeed. Thus Lemma 5 implies that
$${\mathcal N}(f,v) > Dh^{-1}c^52^n = A2^{(2-2\lambda )n} \,$$
whenever
$$
|\eta| 2^n  =|v| \mu \le V2^n = 2^{-5}c^32^n =2^{-5}2^{-3}2^n= 2^{-8}2^n
$$
and hence we can choose $B=2^{-8}$
in the theorem if $n$ is sufficiently large.

\section{Concluding remark}

A very first version by a subset of the authors of the present paper gave weaker bounds for
the (finite) correlation coefficients using a less precise method. We hope to revisit these questions
in the near future, with a detailed discussion about the different correlation coefficients that can
be defined in this context.

\section*{Acknowledgement}{The authors wish to acknowledge the anonymous reviewer for his/her reviewing the paper, including the computational
details, thoroughly.}

\end{document}